\documentclass[12pt,reqno]{amsart}
\usepackage{soul}
\usepackage{hyperref}
\usepackage{multicol,amsmath,graphics,cancel,soul,color,bbm,amsfonts,dsfont,tikz,mathrsfs,amssymb}
\usepackage[left=1in, right=1in, top=1.1in,bottom=1.1in]{geometry}
\usetikzlibrary{matrix,patterns,positioning}
\numberwithin{equation}{section}

\newcommand{\Fc}{\mathcal{F}}

\newcommand{\Pc}{\mathcal{P}}

\newcommand{\E}{\mathbb{E}}

\newcommand{\N}{\mathbb{N}}
\newcommand{\Pb}{\mathbb{P}}

\newcommand{\R}{\mathbb{R}}

\newcommand{\Indi}[1]{\mathbbm{1}_{#1}}

\newcounter{dummy} \numberwithin{dummy}{section}

\newtheorem{Theorem}[dummy]{Theorem}

\newtheorem{Lemma}[dummy]{Lemma}
\newtheorem{Example}[dummy]{Example}
\newtheorem{Remark}[dummy]{Remark}

\def\1{{\rm l}\hskip -0.21truecm 1}

\newcommand{\vertiii}[1]{{\left\vert\kern-0.25ex\left\vert\kern-0.25ex\left\vert #1 
    \right\vert\kern-0.25ex\right\vert\kern-0.25ex\right\vert}}

\def\R{\mathbb{R}}

\def\NN{\mathbb{N}}

\def\PP{\mathbb{P}}
\def\EE{\mathbb{E}}
\def\Var{\mathbb{V}\mathrm{ar}}

\def\al{\alpha}
\def\be{\beta}
\def\ga{\gamma}
\def\de{\delta}

\def\ep{\varepsilon}

\def\wt{\widetilde}

\def\dtv{d_\mathrm{TV}}

\def\dh2l{\mathbf{d}_{\mathbb{H}_{2\ell}}}
\def\d2{\mathbf{d}_2}

\def\cP{\mathcal{P}}

\begin{document}

\title[A generalized Kubilius-Barban-Vinogradov bound for prime multiplicities]{A generalized Kubilius-Barban-Vinogradov  bound for prime multiplicities}
\author{Louis H.Y. Chen, Arturo Jaramillo, Xiaochuan Yang}
\address{Louis H. Y. Chen: Department of Mathematics, National University of Singapore, Block S17, 10 Lower Kent Ridge Road, Singapore 119076.}
\email{matchyl@nus.edu.sg}
\address{Arturo Jaramillo: Centro de Investigación en Matem\'aticas, Jalisco S/N, Col. Valenciana 36023 Guanajuato, Gto.}
\email{jagil@cimat.mx}
\address{Xiaochuan Yang: Department of Mathematics, Brunel University London, Uxbridge UB83PN, United Kingdom}
\email{xiaochuan.j.yang@gmail.com}

\date{November 14, 2021}

\begin{abstract}
We present an assessment of the distance in total variation of \textit{arbitrary} collection of prime factor multiplicities of a random number in $[n]=\{1,\dots, n\}$ and a collection of independent geometric random variables. More precisely, we impose mild conditions on the probability law of the random sample and the aforementioned collection of prime multiplicities, for which a fast decaying bound on the distance towards a tuple of geometric variables holds. Our results generalize and complement those from Kubilius et al. \cite{Kub,BaVi,Ten} which consider the particular case of uniform samples in $[n]$ and collection of ``small primes''. As applications, we show a generalized version of the celebrated Erd\"os Kac theorem for not necessarily uniform samples of numbers.
\end{abstract}
\maketitle

\section{Introduction}

\subsection{Overview}
Let $\Pc$ denote the set of prime numbers. 
For a given $n\in\N=\{1,2,...\}$, we consider a random variable $J_{n}$ supported on $[n]$. 
The goal of this paper is to study asymptotic properties (as $n$ tends to infinity) of the $p$-adic valuations of $J_n$, 
denoted by $\{{v^n_p}\ ;\ p\in\Pc\}$ and characterized by the identity
\begin{align*}
J_{n}
  &=\prod_{p\in\Pc }p^{{v^n_p}}.
\end{align*} More precisely, we will determine general conditions over $J_n$ that will allow us to approximate the law of a random vector of the form 
$\pmb{v}^n=({v^n_p}\ ;\ p\in\Gamma_n)$, where $\Gamma_n$ is a finite subset of $\Pc$, whose cardinality satisfies suitable 
growth conditions. As applications of our results, we will prove a version of Erd\"os-Kac theorem, valid for non-necessarily uniform samples, as well as a Poisson point process approximation for the configuration of non-trivial multiplicities of small primes (see Section \ref{sec:applications} for details). Our approach relies on the use of Bonferroni inequalities and is completely probabilistic; although the the applications will 
typically make use of elementary results from number theory, such as Merten's formula. Thoughout the paper, $[x]$ denotes the set $\{1,...,\lfloor x\rfloor \}$ for any $x>0$ and $\{a|b\}$ denotes the event that $a$ divides $b$ for $a,b\in\NN$. We write $v_p=v^n_p$ for simplicity.\\

\noindent \textit{Motivations and the uniform distribution case}\\
Since the influential manuscript \cite{Kub} by Kubilius in 1962, the use of the random vector $\pmb{v}^n$ 
as a tool for studying divisibility properties of $J_n$ has gained particular traction in probabilistic number theory, 
as these objects naturally emerge in the study of 
arithmetic additive functions of uniform samples. 
Our manuscript takes \cite{Kub} as starting point. There it was proved that in the particular case where 
$J_n$ has uniform distribution over $[n]$ and $\Gamma_{n}=\Pc\cap[n^{1/\be_{n}}]$ 
for a sequence $\be_n>0$ converging to infinity, there exist constants $C,\delta>0$ such that 

\begin{align}\label{eq:Kubiliusorig}
d_{TV}(\pmb{v}^n,{\pmb{g}}^n)
  &\leq Ce^{-\delta \be_n},
\end{align}
where ${\pmb{g}}^n=(g_p\ ;\ p\in\Gamma_n)$ is a random vector whose entries are independent geometric random variables with 
\begin{align*}
\Pb[g_p=k]
  &=p^{-k}(1-p^{-1}),
\end{align*}
for $k\in\N_0:=\mathbb{N}\cup\{0\}$. The aforementioned bound for $d_{TV}(\pmb{v}^n,{\pmb{g}}^n)$, combined with elementary 
probabilistic tools, leads to very powerful results in probabilistic number theory. To exemplify this, 
we would like to mention that \eqref{eq:Kubiliusorig} can be used to obtain a quantitative assessment of the rate of convergence 
in the celebrated Erd\"os-Kac theorem, which establishes the asymptotic normality for the number of 
prime divisors of a uniform sample of $[n]$. The precise statement of Erd\"os-Kac theorem,
as well as its link to the inequality \eqref{eq:Kubiliusorig} and some further generalizations 
will be presented to Section \ref{sec:applications}.\\

\noindent Since their publication, the results from \cite{Kub}
have been extended and generalized in several directions. Next we briefly mention some of them. In Barban et.al. \cite{BaVi} (see also 
\cite[Chapter~3]{Ell}), the 
bound \eqref{eq:Kubiliusorig} was improved to 
\begin{align}\label{eq:barban}
d_{TV}(\pmb{v}^n,{\pmb{g}}^n)
  &\leq C(e^{-\frac{1}{8}\beta_n\log(\beta_n)} + n^{-\frac{1}{15}}),
\end{align}
for some (possibly different) constant $C>0$. In the subsequent papers \cite{Ell} by Elliot and \cite{ArrBa} by Arratia et al.  it was proved that if 
the left hand side of \eqref{eq:Kubiliusorig} converges to zero, then necessarily $\lim_{n}\beta_n=\infty$. We would also like to refer 
the reader to \cite{ArrTav} and \cite{ArrSt} where an analysis of the case of the case where $\beta_{n}=u$ for all $n\in\N$ was carried. In 
such instance it was shown that 
\begin{align}\label{eq:limitisH}
\lim_{n}d_{TV}(\pmb{v}^n,{\pmb{g}}^n)
  =H(u),
\end{align}
where $H$ is defined by 
\begin{align*}
H(u)
  &:=\frac{1}{2}\int_{\R}|\mathcal{B}(v)-e^{-\gamma}|\varrho(u-v)dv+\frac{1}{2}\varrho(u),
\end{align*}
and $\mathcal{B}, \varrho, \gamma$ denote the Buschstav's function, Dickman's function and Euler's constant respectively. 
Finally, we would like to mention the paper \cite{Ten} by Tenenbaum, where it was proved that for every $\varepsilon>0$, 
there exists a constant $C_{\varepsilon}>0$ such that
\begin{align*}
d_{TV}(\pmb{v}^n,{\pmb{g}}^n)
  &\leq C_{\varepsilon}(\varrho(\beta_n)2^{(1+\varepsilon)\beta_n}+n^{-1+\varepsilon}).
\end{align*}
It was proved as well in \cite{Ten} that in the particular case where
$\exp\{(\log\log(n))^{2/5+\varepsilon}\}\leq \beta_n\leq n$, the exact rate
\begin{align}\label{eq:Tenenbaum22}
|d_{TV}(\pmb{v}^n,{\pmb{g}}^n)
-H(\beta_{n})|= o(H(\be_n)),
\end{align}
holds, and for every $\varepsilon>0$,
\begin{align}\label{eq:Tenenbaum2}
d_{TV}(\pmb{v}^n,{\pmb{g}}^n)
  &\leq C_{\varepsilon}(e^{-\be_n \log(\beta_n)}+n^{-1+\varepsilon}).
\end{align}

\noindent 
The above results give a very complete picture of the behavior of $d_{TV}(\pmb{v}^n,{\pmb{g}}^n)$, for the case where 
$J_n$ is uniform in $[n]$ and the primes under consideration are all small, in the sense that $\Gamma_n=\Pc\cap[n^{1/{\be_n}}]$.  
There has however been only little investigation into the problem of giving explicit rates of $d_{TV}(\pmb{v}^n,{\pmb{g}}^n)$ 
for more general choices for $\Gamma_n$ and $J_n$. In this paper we give a partial answer to this, as we address the following questions
\begin{enumerate}
\item[-] Up to what extent, the above results remain valid for different choices of $\Gamma_n$?
\item[-] How much flexibility do we have for choosing the distribution of $J_n$? 
\end{enumerate}
We would like to emphasize that the cases where $\Gamma_n$ contains large primes are of special interest, as the 
complementary case is nearly completely described by the relation \eqref{eq:Tenenbaum22}.\\

\noindent \textit{Some heuristic considerations}\\
\noindent In order to give an exploratory description of the nature of bounds of the type \eqref{eq:Kubiliusorig}, we introduce some suitable 
notation. For any ${D}\subset \Gamma_n$ and $m=(m_{p}\ ;\ p\in{D})\in \NN^{{D}}$, set $|m| := \sum_{p\in{D}} m_p$, $|{D}|:=\sharp{D}$, 
$m+1:=(m_{p}+1\ ;\ p\in{D})$ and
\begin{align*}
p_{D} = \prod_{p\in {D}} p, \quad \quad \quad \quad  p_{D}^m = \prod_{p\in{D}} p^{m_p}
\end{align*}
with the convention that $p_{D}=p_{D}^m=1$ if ${D}=\emptyset$. Define as well the events 
\begin{align*}
{A}({D}, m) 
  &= \{{v_p}=m_p\ \text{ for all } \  p\in{D}\ \text{ and } \   \xi_q=0\ \text{ for all } \   q\in\Gamma_n\setminus {D}\}, \\
\tilde{{A}}({D}, m)  
  &= \{g_p=m_p\ \text{ for all } \   p\in{D}\ \text{ and } \   {g_q}=0\ \text{ for all } \   q\in\Gamma_n\setminus {D}\}.
\end{align*}
This way, we can write
\begin{align}\label{eq:forTV}
d_{TV}(\pmb{v}^n,{\pmb{g}}^n)
=\frac{1}{2}\sum_{{D}\subset\Gamma_n} \sum_{m\in\NN^{{D}}} | \PP[ {A}({D}, m) ] -  \PP[ \tilde{{A}}({D}, m) ] |.
\end{align}
 Formula \eqref{eq:forTV} reduces the problem to estimating 
$\PP[ {A}({D}, m) ]$  and $ \PP[ \tilde{{A}}({D}, m) ] $. One way of doing this, is 
decomposing the events ${A}({D}, m)$ and $\tilde{{A}}({D}, m) $ as unions and intersections of ``elemental events'', 
having the property that their probabilities are easy to approximate. Throughout this paper, the elemental events 
that will serve for approximating $\PP[ {A}({D}, m) ]$ will consist of the elements of the $\pi$-system
\begin{align*}
\mathfrak{E}
  &:=\{\{{v_{p_1}}\geq m_1,\dots, {v_{p_r}}\geq m_{r}\} ;\  p_1,\dots, p_r\in\Pc \text{ and } m_1,\dots, m_r\in{\N}\}\\
	&=\{\{p_1^{m_1}\cdots p_{r}^{m_{r}}| J_n\} ;\  p_1,\dots, p_r\in\Pc \text{ and } m_1,\dots, m_r\in{\N}\},
\end{align*}
while those used to approximate $\PP[ \tilde{A}({D}, m) ]$ will consist on the elements of the $\pi$-system
\begin{align*}
\tilde{\mathfrak{E}}
  &:=\{\{{g_{p_1}}\geq m_1,\dots, g_{p_r}\geq m_{r}\} ;\  p_1,\dots, p_r\in\Pc \text{ and } m_1,\dots, m_r\in{\N}\}.
\end{align*}
This choice is justified by the fact that for every $p_1,\dots, p_r\in\Pc$ satsifying $p_i\neq p_j$ and 
$m_1,\dots, m_{r}\in{\N}$, 
\begin{align}\label{eq:probpisys}
\Pb[{g_{p_1}}\geq m_1,\dots, g_{p_r}\geq m_{r}]
  &=\frac{1}{d},
\end{align}
where $d:=\prod_{i=1}^rp_i^{m_i}.$ In order for our heuristic to be accurate, we are required to impose a condition that 
guarantees that when the variables ${g_p}$ appearing in \eqref{eq:probpisys} are replaced by 
$v_{p}$, the associated probability remains approximately equal to $\frac{1}{d}$. Motivated by this, we will assume that 
$J_{n}$ and $\{g_p\ ;\ p\in\Pc\}$ are defined in a common probability space $(\Omega,\Fc,\Pb)$ and the following 
hypothesis holds:\\

\noindent $(\mathbf{H}_t)$ There exist finite constants $\kappa\geq 1, t>0$ independent of $n$, such that $J_n$ is supported in $[0, n]$ and 
for every $a\in\N$,
\begin{align*}
|\Pb[a\ |\ J_{n}]-\frac{1}{a}| \leq\frac{\kappa}{n^t}  \quad \mbox{ and } \quad \PP[a\ | \ J_n] \le \frac{1+\kappa}{a}.
\end{align*}
Despite the fact that condition $(\mathbf{H}_t)$ imposes a mildly rigid condition over the law of $J_{n}$, it does include the following rich family of probability laws as particular instances.
\begin{Example} 
\label{ex}
Consider the truncated ``Pareto type'' distribution
\begin{align*}
\pi_{n,s}(k) =  \frac{1}{Z_{n,s}} k^{-s}  \quad k\in [n], \mbox{ and }  s\in [0,1), 
\end{align*}
where  $Z_{n,s}= \sum_{k\in[n]}k^{-s}$. We claim that $J_n\sim \pi_{n,s}$ satisfies $(\mathbf{H}_{1-s})$ with $\kappa=3$, for all $n$ large. Indeed, we have
\begin{align}\label{e:div_pareto}
|\PP[a|J_n] - a^{-1}| = \frac{1}{Z_{n,s}}   \Big| \sum_{k=1}^{\lfloor n/a\rfloor} (ak)^{-s} -  a^{-1} \sum_{k=1}^n k^{-s} \Big|
\end{align}
Notice that for any $a\in[n]$,
 \begin{align*}
 \sum_{k=1}^{\lfloor n/a\rfloor} (ak)^{-s} -  a^{-1} \sum_{k=1}^n k^{-s} \le \int_0^{n/a} (ax)^{-s} dx - a^{-1} \int_1^{n+1} x^{-s} dx \le a^{-1} \int_0^1 x^{-s}dx\le (1-s)^{-1}.
 \end{align*}
Similarly, for any $a\in [n]$,
 \begin{align*}
 a^{-1} \sum_{k=1}^n k^{-s} -   \sum_{k=1}^{\lfloor n/a\rfloor} (ak)^{-s}&\le a^{-1} \int_0^{n-1} x^{-s} ds -  \int_1^{n/a} (ax)^{-s} dx  \le a^{-1} \int_0^a x^{-s} dx \le 1+ (1-s)^{-1}.
 \end{align*}
It remains to bound from below $Z_{n,s}$. We have
\begin{align*}
Z_{n,s}\ge \int_1^{n+1} x^{-s} ds = (1-s)^{-1} (n^{1-s}-1)
\end{align*}
Plugging these estimates in \eqref{e:div_pareto} proves that the first bound of $(\mathbf{H}_{1-s})$ holds.  The other bound follows from analogous integral approximation argument.
 \end{Example}
 
\begin{Remark}
Distributions considered in the previous example are not asymptotically uniform in the sense that their total variation distance as $n\to\infty$ does not converge to 0. To verify this, we simply compute the sum
\begin{align*}
\sum_{k\in[n/M]} \Big| \frac{k^{-s}}{Z_{n,s}}  - \frac{1}{n} \Big|, \quad M>1.
\end{align*} 
For any $s\in[0,1)$, by choosing $M$ large enough, we see that the summands are bounded from below by $c/n$, where $c>0$ depends only on $s$ and $M$. This gives a lower bound $c/M>0$ for the total variation distance between the truncated Pareto distributions and the uniform distribution over $[n]$.  
Therefore, the conclusions of our main result cannot be derived from those in the literature for a uniform distribution in $[n]$. 
\end{Remark}

\begin{Example}\label{ex2}
Let $\upsilon:[0,\infty)\rightarrow\R_{+}$ be a continuously differentiable function that is monotone over an interval of the form $[M,\infty)$, for some $M\in\N $, with
\begin{align*}
\sup_{n>M}n^t\int_{0}^1|\upsilon(sn)-L|ds<\infty,
\end{align*}
for some $L>0$ and $t>0$. If the law of $J_{n}$ is supported in $\{1,\dots, n\}$ and satisfies 
\begin{align*}
\Pb[J_{n}=k]
  &:=c_{n}\upsilon(k),
\end{align*}
for some $k=1,\dots, n$ and $c_{n}>0$, then the variables $J_{n}$ satisfy condition $(\mathbf{H}_t)$ for $n$ sufficiently large. To verify this, observe that for every natural number $a\in\N$,
\begin{align}\label{eq:adividesJnexact}
\Pb[a\ |\ J_{n}]
  &= n^tc_n\left(\frac{1}{n^t}\sum_{j=1}^{\lfloor n/a\rfloor}\upsilon(aj)\right).
\end{align}
By elementary algebraic manipulations, 
\begin{align*}
\frac{1}{n^t}\sum_{j=1}^{\lfloor n/a\rfloor}\upsilon(aj)
  &=O(1/n^t)+\frac{1}{n^t}\sum_{j=M}^{\lfloor n/a\rfloor}\upsilon(aj)
	=O(1/n^t)+\frac{1}{n}\int_{M}^{ n/a}\upsilon(as)ds=O(1/n^t)+L/a,
\end{align*}
where $O(1/n^t)$ denote error functions bounded in absolute value by a constant multiple of $1/n^t$, independent of $a$. By choosing $a=1$, we get that $n^tc_{n}=L+O(1/n^t)$. Condition $(\mathbf{H}_t)$ then follows from \eqref{eq:adividesJnexact}. Some particular instances in which the above conditions hold are the case where $J_{n}$ has uniform distribution over $\{1,\dots, n\}$ and more generally, the case where there exist a non-increasing function $\theta:\R_{+}\rightarrow\R_{+}$, with
$$\sup_{n\geq 1}k^t\theta(k)<\infty,$$
as well as constants $\varepsilon>0$ and $c_{n}>0$ such that $\Pb[J_{n}=k]=c_{\alpha,n}(\varepsilon+\theta(k))$ for all $k=1,\dots, n$ and $\Pb[J_{n}=k]=0$ otherwise.\\
\end{Example}

\begin{Example} 
A generalization of the above two examples gives rise the following large family of probability distributions satisfying  $(\mathbf{H}_t)$.  Let $\upsilon:[0,\infty)\rightarrow\R_{+}$ be a continuously differentiable function that is monotone over an interval of the form $[M,\infty)$, for some $M\in\N $, with
\begin{align*}
\sup_{n>M}\frac{1}{c_n}\int_{0}^1 |\upsilon(sn)|ds<\infty,
\end{align*}
for some $L>0$ and $t\in(0,1)$, where 
\begin{align*}
c_n
  :=\left(\sum_{k=1}^n\upsilon(k)\right)^{-1}.
\end{align*}
If the law of $J_{n}$ is supported in $\{1,\dots, n\}$, $c_n$ has a decay of the order $n^{-t}$ and 
\begin{align*}
\Pb[J_{n}=k]
  &:=c_{n}\upsilon(k),
\end{align*}
for some $k=1,\dots, n$ and $c_{n}>0$, then the variables $J_{n}$ satisfy condition $(\mathbf{H}_t)$ for $n$ sufficiently large. The proof of this claim is identical to that of Example \ref{ex2}, with the exception that $n^t$ should be replaced by $c_n^{-1}$ and $L$ should be replaced by zero.
\end{Example}

\noindent Notice that 
\begin{align*}
{A}({D}, m)
  &=\{{v_p}\geq m_p\ ;\ p\in{D}\}\backslash\bigcup_{p\in\Pc}\{{v_p}\geq m_{p}+1\ ;\ p\in{D}\}\cup\{\xi_q\geq 1\ ;\ q\in {D}^c\}\\
\tilde{{A}}({D}, m)
  &=\{g_p\geq m_p\ ;\ p\in{D}\}\backslash\bigcup_{p\in\Pc}\{g_p\geq m_{p}+1\ ;\ p\in{D}\}
	\cup\{{g_q}\geq 1\ ;\ q\in{D}^c\},
\end{align*}
where ${D}^c$ denotes the complement relative to $\Gamma_n$, namely, 
${D}^c:=\Gamma_n\backslash{D}$. Consequently, by an elementary application of the inclusion-exclusion principle, 
\begin{align*}
\PP[{A}({D}, m)]
  &= \sum_{ I\subset \Gamma_n} (-1)^{|I|} \Pb[\{p_{{D}}^m|J_n\}\cap \{p_{{D}\cap I}^{m+1} p_{{D}^c\cap I}|J_n\}]\\
  &= \sum_{ I\subset \Gamma_n} (-1)^{|I|} 
	\Pb[p_{{D}\cap I^c}^mp_{{D}\cap I}^{m+1} p_{{D}^c\cap I}|J_n]
	= \sum_{ I\subset \Gamma_n} (-1)^{|I|} \Pb[p_{{D} }^m p_{I}|J_n]
\end{align*}
for any ${D}\subset \Gamma_n$ and $m\in\NN^{D}$. Similarly, 
\begin{align*}
\PP[\tilde{{A}}({D}, m)]
  &= \sum_{I\subset \Gamma_n}  \frac{ (-1)^{| I |}}{p_{{D} }^{m } p_{I}}.
\end{align*}
Thus, by using the hypothesis $(\mathbf{H}_t)$, we obtain the bound  
\begin{align*}
d_{TV}(\pmb{v}^n,{\pmb{g}}^n)
&\leq\frac{1}{2}\sum_{{D}\subset\Gamma_n} \sum_{m\in\NN_{+}^{{D}}}\sum_{ I\subset \Gamma_n} 
(\Indi{\{p_{{D} }^{m } p_{I}\leq n\}}\frac{1}{n}+\Indi{\{p_{{D} }^{m } p_{I}>n\}}
\frac{1}{p_{{D} }^{m } p_{I}}).
\end{align*}	
By using the geometric sum formula for the second term and the fact that there are at most $\frac{1}{2}\log(n)$ values for $m$ for which 
$p_{{D} }^{m } p_{I}\leq n$, we deduce that there exists a universal constant $C>0$, independent of $\Gamma_n$ or $n$, such that
\begin{align}\label{eq:roughbTV}
d_{TV}(\pmb{v}^n,{\pmb{g}}^n)
&\leq C\sum_{{D}\subset\Gamma_n}\sum_{ I\subset \Gamma_n}\frac{\log(n)}{n}
\leq  C4^{ |\Gamma_n|}
\frac{\log(n)}{n}.
\end{align}	
In particular, when $\sup_n|\Gamma_n|<\infty$, one deduces that $d_{TV}(\pmb{v}^n,{\pmb{g}}^n)$ is of the order 
$\frac{\log(n)}{n}$ . Naturally, one wonders if in the case where $|\Gamma_{n}|$ converges to infinity, 
the above argument leads to a ``good'' bound for $d_{TV}(\pmb{v}^n,{\pmb{g}}^n)$.
Unfortunately, the right hand side of \eqref{eq:roughbTV} diverges if $|\Gamma_n|$ is asymptotically larger than $\frac{\log(n)}{\log(4)}$, 
which restricts significantly the possible choices of $\Gamma_n$.\\
 
\noindent In this paper, we make suitable adjustments to the argument above, so that we obtain a bound that converges to zero even when 
$|\Gamma_n|$ is polynomial of any degree in $\log(n)$, see Remark \ref{r:sizeGamma} for more details. The main idea
consists of replacing the use of the inclusion-exclusion principle by Bonferonni-type bounds. 
This, combined with a careful combinatorial analysis of the resulting terms, 
leads to a near to optimal bound for $d_{TV}(\pmb{v}^n,{\pmb{g}}^n)$.

\subsection{Statement of the main results}
Recall that $J_n$ is assumed to satisfy the condition $(\mathbf{H}_t)$. To state the main theorem, we introduce the quantities
\begin{align*}
\tau_n:=\sum_{p\in\Gamma_n} \frac{1}{p}  \quad\quad \mbox{ and } \quad\quad \rho_n:= \frac{\log n}{\log |\Gamma_n|}.
\end{align*}

\begin{Theorem}\label{t}
Suppose that $\lim_{n\to\infty} \rho_n =\infty$ and that there exists $\varepsilon>0$, such that 
$|\Gamma_n|\leq n^{ \tau_{n}^{-1-\varepsilon} }$ for all $n\in\NN$. Suppose that $(\mathbf{H}_t)$ holds with $t>0$. Then for $n$ sufficiently large, we have 
\begin{align}\label{eq:GenKub}
\dtv(\pmb{v}^n,{\pmb{g}}^n)
  &\leq (7+4\kappa) \exp\Big( - c \min\big( \rho_n \log\rho_n, \log(n) \big) \Big)
\end{align}
with $c= \frac{t(1\wedge\ep)}{12(1+\ep)}$.
\end{Theorem}

\begin{Remark}
In the case where $\Gamma_{n}=\Pc\cap[n^{\frac{1}{\beta_n}}]$ for $\beta_n>1$ converging to infinity, 
we have that $\lim_{n}n^{-\frac{1}{\beta_n}}\log(n^{\frac{1}{\beta_n}})|\Gamma_n|=1$ due to the prime 
number theorem. From here it follows that $\lim_{n}\frac{\rho_n}{\beta_n}=1$. Consequently, up to a change in the constant 
$c$, Theorem \ref{t} improves Barban's bound \eqref{eq:barban} when $\varepsilon$ is large enough and
\begin{align}\label{eq:asumpbeta}
\beta_n\log(\beta_n)\leq \log(n),
\end{align}
for $n$ sufficiently large. 
 
\end{Remark}

\begin{Remark}
If $\rho_n$ is uniformly bounded over $n$, then $|\Gamma_n|\ge n^{a}$ for some fixed $a\in(0,1)$ independent of $n$. 
In this case, by \eqref{eq:limitisH},
%
in order for a bound for the left hand side of \eqref{eq:GenKub} to converge to zero at the time that it generalizes 
the bounds presented by Kubilius 
and Barban bounds, we have to restrict ourselves to the regime $\lim_{n\to\infty} \rho_n =\infty$. 
This justifies the appearance of the condition $\lim_{n\to\infty} \rho_n =\infty$ as a hypothesis in Theorem \ref{t}. 
However, a bound for case $\sup_{n}\rho_n<\infty$ can be obtained by means of Equation \eqref{eq:roughbTV}.
\end{Remark}

\begin{Remark}\label{r:sizeGamma}
By Mertens' formula (see \cite[page 14]{Ten}), we have that for $n$ sufficiently large, $\tau_n\le \log\log(n)+1$. Consequently, 
for any $K>0$, we have that
\begin{align*}
n^{\tau_n^{-1-\ep}} \ge e^{ \frac{\log(n)}{(1+\log\log(n))^{1+\ep}} } \ge [\log(n)]^K,
\end{align*}
provided that $n$ is sufficiently large. Thus, our condition on the cardinality of $\Gamma_n$ is mild, as one 
only requires the condition $|\Gamma_n|\leq n^{ \frac{1}{(1+\log\log(n))^{1+\ep}} }$ to be satisfied. 
\end{Remark}

\begin{Remark}
By considering $\Gamma_n$ of small cardinality, we observe that one cannot surpass the barrier $\log(n)$ in the exponent. 
For instance, if
\begin{align*}
|\Gamma_n| \le \log\log(n),
\end{align*}
then we have $\rho_n\log(\rho_n) \ge \log(n)\log\log(n)$ for $n$ sufficiently large, and any estimate of the type 
$c_1\exp(- c_2 \rho_n\log(\rho_n) )$ decays faster than $c_1 e^{-c_2\log(n)\log\log(n)}$. 
The former quantity can't bound $\dtv(\pmb{v}^n, {\pmb{g}}^n)$ since for every $p\in\Pc$,
\begin{align*}
\dtv(\pmb{v}^n, {\pmb{g}}^n) &= \frac{1}{2}\sum_{m\in\NN_0^{\Gamma_n}} |\PP [ \pmb{v}^n=m] - \PP[\pmb{g}^n=m]| 
\ge \frac{1}{2} \PP\Big[g_p\ge \frac{\log(n)}{\log(p)} \Big],
\end{align*}
so that 
\begin{align*}
\dtv(\pmb{v}^n, {\pmb{g}}^n) \ge p^{ - 1 - (\log(n)/\log(p))}=\frac{1}{pn}. 
\end{align*}
\end{Remark}

\section{Proof of theorem \ref{t}}
The main idea consists on decomposing the right hand side of
 \eqref{eq:forTV} into three pieces, 
which heuristically correspond to the following instances
\begin{enumerate}
\item[-] The set ${D}\subset\Gamma_n$ of prime divisors  has large cardinality, see Section  \ref{s:large_D}.
\item[-] The set ${D}\subset\Gamma_n$ is small, but the associated multiplicities $m\in\N^{{D}}$ are large, see Section \ref{s:small_D_large_m}.
\item[-] Both the set ${D}\subset\Gamma_n$ and the multiplicities $m\in\N^{{D}}$ are small, see Section \ref{s:small_D_small_m}.
\end{enumerate}
A suitable control for each of these instances will be considered in Lemmas \ref{l:many},\ref{l:high}, \ref{l:bonf}. 

\subsection{Many distinct prime divisors}\label{s:large_D}
First we consider ${D}$ with large cardinality. 
We let the threshold $\al_{n,\delta}$ of the form 
\begin{align}\label{eq:alphadef}
\alpha_{n,\delta} := \delta \rho_n,
\end{align}
where $\delta>0$ is a positive constant independent of $n$. For convenience in the notation, we will avoid specifying the dependence of 
$\alpha_{n,\delta}$ on $\delta$ and simply write $\alpha_n=\alpha_{n,\delta}$.

\begin{Lemma}\label{l:many}  Suppose that the second bound of $(\mathbf{H}_t)$ holds. 
For any $\de>0$ and $n\in\NN$, we have
\begin{align*}
&\sum_{\substack{{D}\subset\Gamma_n\\|{D}|\ge \al_n}}
\sum_{m\in\N^{{D}}} |\PP[{A}({D}, m)]-\PP[\tilde{{A}}({D}, m)] |\\
  &\leq (2+\kappa)\exp\Big( -\frac{\de\ep}{1+\ep} \rho_n \log(\rho_n) + \de(1-\log(\de))\rho_n  \Big).
\end{align*}
\end{Lemma}
\begin{proof}
The sum on the left-hand side is empty if $\al_n>|\Gamma_n|$ in which case the claim is trivial. Suppose that  $\al_n\le |\Gamma_n|$.  We simply bound
$$|\PP[\tilde{{A}}({D}, m)] - \PP[{A}({D}, m)] |\leq \PP[\tilde{{A}}({D}, m)] + \PP[{A}({D}, m)]$$ and estimate the sum of probabilities separately. Notice that 
\begin{align*}
\sum_{\substack{{D}\subset\Gamma_n\\|{D}|\ge \al_n}}
\sum_{m\in\N^{{D}}}\PP[{A}({D}, m)]
  &=\sum_{\substack{{D}\subset\Gamma_n\\|{D}|\ge \al_n}} \PP[{v_p} \ge 1, \forall p\in {D}, \xi_q = 0, \forall q\in \Gamma_n\setminus{D}] \\
  &\le \sum_{\substack{{D}\subset\Gamma_n\\|{D}|\ge \al_n}}\PP[{v_p} \ge 1, \forall p\in {D}]
	= \sum_{\substack{{D}\subset\Gamma_n\\|{D}|\ge \al_n}}
	\Pb[p_D|J_n] 
	\le \sum_{\substack{{D}\subset\Gamma_n\\|{D}|\ge \al_n}} \frac{1+\kappa}{p_{D}},
\end{align*}
where the last inequality follows from hypothesis $(\mathbf{H}_t)$. Observe that the condition 
$|\Gamma_n|\leq n^{\tau_n^{-(1+\varepsilon)}}$ implies that 
$\log(|\Gamma_n|)\leq \frac{\log(n)}{\tau_n^{1+\varepsilon}}$, which leads to the bound 
\begin{align}\label{eq:technical}
\rho_{n}\geq \tau_{n}^{1+\varepsilon}.
\end{align}
 Thus, by using Chernoff's bound 
(see Lemma \ref{l:chernoff}), 
\begin{align}\label{e:p5}
 \sum_{\substack{{D}\subset\Gamma_n\\|{D}|\ge \al_n}}  \frac{1}{p_{D}}
= \sum_{j=\al_n}^{|\Gamma_n|} \frac{1}{j!} \sum_{(p_1,...,p_j)\in (\Gamma_n)^j_{\neq}} \frac{1}{p_1\cdots p_j}
\le \sum_{j=\al_n}^\infty \frac{\tau_n^{j}}{j!}\le  \Big(\frac {e\tau_n}{\al_n}\Big)^{\al_n},
\end{align}
where $(\Gamma_n)^j_{\neq}$ is the set of $j$-tuples of elements in $\Gamma_n$ with distinct coordinates. We observe by \eqref{eq:technical} that 
\begin{align*}
\log\Big(\Big(\frac {e\tau_n}{\al_n}\Big)^{\al_n}\Big)
    &\le \de\rho_n\Big(1+ \frac{1}{1+\ep} \log(\rho_n) - \log(\de) - \log(\rho_n)\Big) \\
    &\le -\frac{\de\ep}{1+\ep} \rho_n \log(\rho_n) + \de(1-\log(\de))\rho_n,
\end{align*} 
yielding the desired estimate for the  sum of $\PP[{A}({D},m)]$. 
The same argument gives an analogous bound for the sum of $\PP[\tilde {{A}}({D},m)]$, ending the proof. 
\end{proof}

\subsection{Fewer prime divisors with  overall high multiplicities}\label{s:small_D_large_m}

Now we consider the case  of fewer prime divisors $|{D}|\le \al_n$. Recall that $|m|\ge |{D}|$ since $m\in \NN^{D}$. 
We are interested in the situation where the prime divisors ${D}$ have overall high multiplicities in the sense that 
$|m|\ge \beta_n$ with $\beta_n$ much larger than $\al_n$.

\begin{Lemma}\label{l:high} Suppose that the second bound of $(\mathbf{H}_t)$ holds. 
Let $\alpha_n$ be given by \eqref{eq:alphadef} and define 
\begin{align}\label{eq:betadef}
\be_n := \frac{2\de}{(1+\ep)\log(1.5)}\rho_n \log(\rho_n).
\end{align}
Then we have for all $n\in\NN$,
\begin{align}\label{eq:sumfewplargem}
&\sum_{\substack{{D}\subset\Gamma_n\\|{D}|\le \al_n}} \sum_{\substack{m\in\N^{{D}}\\|m|\ge \beta_n}}
|\PP[\tilde{{A}}({D}, m)] - \PP[{A}({D}, m)] | \\
&\le (2+\kappa) \exp\Big( -\frac{\de}{1+\ep}\rho_n\log(\rho_n) + 5\log(1.5)\de \rho_n + \log(\de\rho_n) \Big). 
\end{align}
\end{Lemma}
\begin{proof}
As in the proof of Lemma \ref{l:high}, we simply estimate the difference of probabilities by their sum. 
We start by handling the terms involving the quantities $\PP[\tilde{{A}}({D}, m)]$. Notice that
\begin{align*}
\sum_{\substack{{D}\subset\Gamma_n\\|{D}|\le \al_n}} \sum_{\substack{m\in\NN^{{D}}\\ |m|\ge \beta_n}} \PP[\tilde{{A}}({D}, m)]
 &\leq \sum_{\substack{{D}\subset\Gamma_n\\|{D}|\le \al_n}}\PP\Big[\sum_{p\in{D}} {g_p} \ge \beta_n \,|\, g_p\ge 1, p\in{D}\Big] \PP\Big[ g_p\ge 1, p\in{D}\Big] \\
 &= \sum_{\substack{{D}\subset\Gamma_n\\|{D}|\le \al_n}}\PP\Big[\sum_{p\in{D}} {\hat g_p} \ge \beta_n \Big] \frac{1}{p_{D}},
\end{align*}
where $\mathcal L({\hat g_p}, p\in {D})=\mathcal L({g_p}, p\in {D} \,|\, g_p\ge 1, p\in{D})$. One readily checks that $({\hat g_p}, p\in{D})$ is a  family of independent $\NN$-valued geometric random variables with 
\begin{align}\label{e:gHat_def}
\PP[{\hat g_p}= k] = p^{-(k-1)}(1-p^{-1}), \quad k\in \NN, p\in{D}.
\end{align}
For any $p\in{D}$, a direct computation leads to the uniform bound $\EE[(1.5)^{{\hat g_p}}]\le 3$ and the following concentration bound
\begin{align*}
\PP\Big[\sum_{p\in{D}} {\hat g_p} \ge \beta_n\Big]\le (1.5)^{-\be_n} 3^{|{D}|}  \le  (1.5)^{-(\be_n-3\al_n)}
\end{align*}
holds by Markov's inequality. This leads to 
\begin{align}
\sum_{\substack{{D}\subset\Gamma_n\\|{D}|\le \al_n}}\sum_{\substack{m\in\N^{{D}}\\|m|\ge \beta_n}} \PP[\tilde{{A}}({D}, m)]
&\le (1.5)^{-(\be_n-3\al_n)} \sum_{j=0}^{\al_n\wedge|\Gamma_n|}  \frac{1}{j!} \sum_{(p_1,...,p_j)\in (\Gamma_n)^j_{\neq}} \frac{1}{p_1\cdots p_j} \notag\\
& \le (1.5)^{-(\be_n-3\al_n)} \sum_{j=0}^{\al_n\wedge|\Gamma_n|}  \frac{\tau_n^j}{j!} 
 \le   \alpha_n  (1.5)^{-(\be_n-3\al_n)} (\tau_n)^{\alpha_n}. \label{e:indep}
\end{align}
Now we move to the analysis of the sums involving $\PP[{A}({D}, m)]$. As before, we write
\begin{align}\label{e:p7}
\sum_{\substack{{D}\subset\Gamma_n\\|{D}|\le \al_n}}\sum_{\substack{m\in\NN^{{D}}\\|m|\ge \beta_n}} \PP[ {A}({D}, m)]
&\le \sum_{\substack{{D}\subset\Gamma_n\\|{D}|\le \al_n}}\sum_{\substack{m\in\NN^{{D}}\\|m|\ge \beta_n}} \PP[{v_p} \ge m_p, \forall p\in {D}] \notag\\
& \leq  \sum_{\substack{{D}\subset\Gamma_n\\|{D}|\le \al_n}}\sum_{\substack{m\in\NN^{{D}}\\|m|\ge \beta_n}} \frac{1+\kappa}{p_{D}^m} \notag\\
& =\sum_{\substack{{D}\subset\Gamma_n\\|{D}|\le \al_n}} \PP\Big[\sum_{p\in{D}} {\hat g_p} \ge \beta_n \Big] \frac{1+\kappa}{p_D}  \prod_{p\in{D}} \frac 1 {1-p^{-1}}.   
\end{align}
where \eqref{e:p7} follows from the definition of $\hat g_p$ at \eqref{e:gHat_def}.
Bounding the product in \eqref{e:p7} by $2^{|{D}|}\le (1.5)^{2\al_n}$ and then using the argument leading to \eqref{e:indep}, we obtain
\begin{align*}
\sum_{\substack{{D}\subset\Gamma_n\\|{D}|\le \al_n}}\sum_{\substack{m\in\NN^{{D}}\\|m|\ge \beta_n}} \PP[ {A}({D}, m)]
\le (1+\kappa)\al_n (1.5)^{-(\be_n-5\al_n)} (\tau_n)^{\al_n}. 
\end{align*}
To summarize, both sums have analogous upper bounds. It follows from \eqref{eq:alphadef}-\eqref{eq:betadef} that 
\begin{align*}
\log\Big(\al_n (1.5)^{-(\be_n-5\al_n)}(\tau_n)^{\al_n}\Big) 
&\le \log(\de\rho_n) - (\be_n-5\de\rho_n)\log(1.5) +  \frac{\de }{1+\ep} \rho_n\log(\rho_n) \\
&\le -\frac{\de}{1+\ep}\rho_n\log(\rho_n) + 5\log(1.5)\de \rho_n + \log(\de\rho_n).
\end{align*}
The conclusion follows immediately. 
\end{proof}

\subsection{Fewer prime divisors with moderate multiplicities}\label{s:small_D_small_m}

Now we handle the remaining case where $|{D}|\le \al_n$ and $|m|\le \be_n$. This is the only part where we seek cancellations between the probability mass functions of independent (i.e. $\pmb{g}^n$) and dependent (i.e. $\pmb{v}^n$) vectors. 
\begin{Lemma}\label{l:bonf}
Let $\al_n$ and $\be_n$ be given by \eqref{eq:alphadef} and \eqref{eq:betadef}  with $\de\in (0,t/3)$. Suppose $(\mathbf{H}_t)$ holds with $t>0$. For $n$ sufficiently large, we have
\begin{align}\label{eq:fewpsmallm}
\sum_{\substack{{D}\subset \Gamma_n\\ |{D}|\le \al_n}} \sum_{\substack{m\in\N^{{D}}\\|m|\le \be_n}} |\PP[{A}({D},m)] - \PP[\tilde{{A}}({D},m)] \le  (3+2\kappa)  \exp(  -c\min[\log(n), \rho_n\log(\rho_n)])
\end{align}
with $c= \min(t-3\de, \frac{\de\ep}{2(1+\ep)})$.

\end{Lemma}

\begin{proof}
We claim that for any positive odd integer $\ga$, the following estimate holds:
\begin{align}\label{l:bonfprev}
&\sum_{\substack{{D}\subset \Gamma_n\\ |{D}|\le \al_n}} \sum_{\substack{m\in\N^{{D}}\\|m|\le \be_n}} |\PP[{A}({D},m)] - \PP[\tilde{{A}}({D},m)]|\nonumber\\
&\le   
 \frac{3 }{\sqrt{\ga}}\Big(\frac{e\tau_n}{\ga}\Big)^\ga  \alpha_n(2e)^{2\tau_n} + 
\begin{cases}
\frac{2(1+\kappa) |\Gamma_n|^{\ga+1}}{n} \al_n e^{2\al_n}  & \mbox{ if } \sqrt{\be_n|\Gamma_n|}\le \al_n, \\
\frac{2(1+\kappa) |\Gamma_n|^{\ga+1}}{n}  \al_n \Big(\frac{e^2\be_n|\Gamma_n|}{\al_n^2}\Big)^{\al_n}& \mbox{ else}.
\end{cases} 
\end{align}
In order to show this, we apply  Bonferonni-type estimates in Lemma \ref{l:bonf_App} for any positive odd integer $1\le \ga\le |\Gamma_n|$, as well as the condition $(\mathbf{H}_t)$, yielding
\begin{align*}
|\PP[{A}({D},m)] - \PP[\tilde{{A}}({D},m)]| \notag &\le \sum_{\substack{I\subset\Gamma_n \\ |I|\le \ga}}  \Big| \PP[p^m_D p_I | J_n]  - \frac{1}{p_{D}^m p_I} \Big| +  4\kappa \sum_{\substack{I\subset \Gamma_n \\ |I|=\ga+1}}  \frac{1}{p_{D}^m p_I}  \notag\\
&\le \frac {2 |\Gamma_n|^{\ga+1}} {n^t}  + \frac{4\kappa}{p^m_{D}}  \frac{\tau_n^\ga}{\ga!}  \notag  \\
 & \le   \frac {2|\Gamma_n|^{\ga+1}} {n^t}+ \frac{6\kappa}{p^m_{D}} \frac{1}{\sqrt{2\pi\ga}}\Big(\frac{e\tau_n}{\ga}\Big)^\ga, 
\end{align*}
where we used Stirling's estimate $n!\ge \sqrt{2\pi n} (n/e)^n$ in the last inequality.  We first show 
\begin{align}\label{e:1}
\sum_{\substack{{D}\subset\Gamma_n\\ |{D}|\le \al_n}}\sum_{\substack{m\in\NN^{D}\\ |m|\le \be_n}}  \frac {2|\Gamma_n|^{\ga+1}} {n^t} \le 
\begin{cases}
\frac{4 |\Gamma_n|^{\ga+1}}{n^t} \al_n e^{2\al_n}  & \mbox{ if } \sqrt{\be_n|\Gamma_n|} \le  \al_n, \\
\frac{4 |\Gamma_n|^{\ga+1}}{n^t}  \al_n \Big(\frac{e^2\be_n|\Gamma_n|}{\al_n^2}\Big)^{\al_n} & \mbox{ else}.
\end{cases} 
\end{align}
Notice that 
\begin{align*}
|\{m\in \NN^{D}: |m|\le \be_n\}| &= \sum_{a=|{D}|}^{\be_n} { a-1\choose |{D}|-1} \le \frac{(\be_n)^{|{D}|}}{|{D}|!}.
\end{align*}
Thus, using Stirling's formula in the last inequality, we have
\begin{align*}
\sum_{\substack{{D}\subset\Gamma_n\\ |{D}|\le \al_n}}\sum_{\substack{m\in\NN^{D}\\ |m|\le \be_n}} 1 \le 1 + \sum_{j=1}^{\al_n} {\Gamma_n \choose j} \frac{\be_n^j}{j!}\le 1 + \sum_{j=1}^{\al_n} \frac{(\be_n|\Gamma_n| )^j}{(j!)^2} \le 1+ \sum_{j=1}^{\al_n} \Big(\frac{e^2\be_n|\Gamma_n|}{j^2}\Big)^j.
\end{align*}
Observe that the summand of the last sum as a function of $j$ increases then decreases as $j$ grows from $1$ to  $\infty$, and it attains the maximum at $j=\sqrt{\be_n|\Gamma_n|}$. If $\sqrt{\be_n|\Gamma_n|}\le \al_n$, we bound the summands from above by $e^{2\sqrt{\be_n|\Gamma_n|}}\le e^{2\al_n}$, otherwise, the summands are bounded by $(\frac{e^2\be_n|\Gamma_n|}{\al_n^2})^{\al_n}$, thus yielding \eqref{e:1}. 

Next we show
\begin{align}\label{e:2}
\sum_{\substack{{D}\subset\Gamma_n\\ |{D}|\le \al_n}}\sum_{\substack{m\in\NN^{D}\\ |m|\le \be_n}}  \frac{6\kappa}{p^m_{D}} \frac{1}{\sqrt{2\pi\ga}}\Big(\frac{e\tau_n}{\ga}\Big)^\ga
\le \frac{6\kappa }{\sqrt{\ga}}\Big(\frac{e\tau_n}{\ga}\Big)^\ga  \alpha_n(2e)^{2\tau_n}.
\end{align}
Manipulating the sum over $m$ as in \eqref{e:p7} gives
\begin{align*}
\sum_{\substack{m\in\NN^{D}\\ |m|\le \be_n}}\frac{1}{p^m_{D}}
= \prod_{p\in{D}} \frac 1 {1-p^{-1}} \PP\Big[\sum_{p\in{D}} g_p \le \be_n, g_p\ge 1, p\in{D}\Big]
\le 2^{|{D}|} \PP[{g_p}\ge 1, p\in{D}] 
=\frac{ 2^{|{D}|}}{p_{{D}}}.
\end{align*}
Hence, handling the sum of $(p_{D})^{-1}$ as in \eqref{e:p5} yields
\begin{align*}
\sum_{\substack{{D}\subset\Gamma_n\\ |{D}|\le \al_n}}\sum_{\substack{m\in\NN^{D}\\ |m|\le \be_n}} \frac{6\kappa}{p^m_{D}}
  & \leq 6\kappa \sum_{j=0}^{\al_n}2^{j}\sum_{\substack{{D}\subset \Gamma_n\\|{D}|=j}}\frac{1}{p_{{D}}}
	\leq 6\kappa \sum_{j=0}^{\al_n}\frac{(2\tau_n)^j}{j!}
	\le 6\kappa \sqrt{2\pi}\sum_{j=0}^{\al_n}\Big(\frac{2e\tau_n}{j}\Big)^j,
\end{align*}
where we used again Stirling's estimate the last inequality. Observe that the summands,  regarded as a function of $j$, attains its maximum at $j=\lfloor 2\tau_n\rfloor$. Hence,  
\begin{align*}
\sum_{\substack{{D}\subset\Gamma_n\\ |{D}|\le \al_n}}\sum_{\substack{m\in\NN^{D}\\ |m|\le \be_n}} \frac{6\kappa }{p^m_{D}}
  & \le 6\kappa\sqrt{2\pi}\alpha_n(2e)^{2\tau_n},
\end{align*}
leading to \eqref{e:2}. Combining \eqref{e:1} and \eqref{e:2} gives \eqref{l:bonfprev}.\\

\noindent It remains to prove \eqref{eq:fewpsmallm}. To this end, we apply \eqref{l:bonfprev}  with $\ga=\ga_n$ given by
\begin{align}\label{e:ga_def}
\ga_n =\max\{ k \le \al_n: k \mbox{ is an odd integer}\}.
\end{align}
 We distinguish i) $\sqrt{\be_n|\Gamma_n|}\le \al_n$ and ii)  $\sqrt{\be_n|\Gamma_n|}> \al_n$. In case i), 
\begin{align*}
\log\Big( \frac{|\Gamma_n|^{\ga_n+1}}{n^t} \al_n e^{2\al_n}\Big) 
&= ( \de\rho_n+ 1 ) \frac{\log(n)}{\rho_n} + \log(\de\rho_n) + 2\de\rho_n - t\log(n)  \\
&\le  -\big[ t-\de+ (\rho_n)^{-1} - 2\de [\log(|\Gamma_n|)]^{-1} - \frac{\log(\de\log(n))}{\log(n)} \big]\log(n)\\
&\le -(t-3\de)\log(n)
\end{align*}
for all $n$ sufficiently large, where we used the condition $|\Gamma_n|\ge e^2$.
  On the other hand, by $\rho_n\ge \tau^{1+\ep}_n$, we have
\begin{align}\label{e:3}
\log\Big( \big(\frac{e\tau_n}{\ga_n}\big)^{\ga_n} \al_n (2e)^{2\tau_n} \Big) &= \de\rho_n\Big(1-\log(\de) - \frac{\ep}{1+\ep} \log(\rho_n)\Big) + \log(\de\rho_n) + 4\rho_n^{(1+\ep)^{-1}} \notag\\
&\le  - \frac{\de\ep}{1+\ep} \rho_n\log(\rho_n) + \Big( 1-\log(\de) + 4\de^{-1} \rho_n^{- \frac{\ep}{1+\ep}} \Big)\de\rho_n + \log(\de\rho_n) \notag\\
&\le - \frac{\de\ep}{2(1+\ep)}\rho_n\log(\rho_n) 
\end{align}
for $n$ sufficiently large, ending the proof of \eqref{eq:fewpsmallm} for case i). 
In case ii), we have
\begin{align*}
&\log \Big( \frac{|\Gamma_n|^{\ga_n+1}}{n^t}  \al_n \Big(\frac{e^2\be_n|\Gamma_n|}{\al_n^2}\Big)^{\al_n} \Big) \\
&= (2\de\rho_n + 1) \frac{\log(n)}{\rho_n} + \log(\de\rho_n)  \\ 
&\quad\quad\quad  +\de\rho_n\Big( 2+ \log\Big[\frac{2\de}{(1+\ep)\log(1.5)}\Big]  +\log(\rho_n \log(\rho_n))  - 2\log(\de\rho_n) \Big) - t \log(n)  \\
&\le - [t-2\de + (\rho_n)^{-1} ] \log(n) - \de\rho_n\log(\rho_n) +  \de\rho_n \log\log(\rho_n) \\
&\quad\quad\quad + \Big(2+\log\Big[\frac{2\de}{(1+\ep)\log(1.5)}\Big]  - 2\log(\de) \Big) \de \rho_n   + \log(\de\rho_n) \\
&\le -(t-3\de)\log(n) 
\end{align*}
for $n$ sufficiently large.  This, together with \eqref{e:3}, ends the proof of \eqref{eq:fewpsmallm} for case ii), thereby ending the proof of the lemma.
\end{proof}

\section{Applications}\label{sec:applications}

The total variation bound obtained in the paper allows to transfer distributional approximation results valid in the realm of independent random variables to the dependent prime multiplicities of a random sample $J_n$, at the price of an extra  error term that appears in Theorem \ref{t}. 

In this section, we mention two such possibilities: a generalized Erd\"os-Kac theorem, 
and a Poisson process approximation result. 

\subsection{Generalized Erd\"os-Kac theorem}

Consider the prime factor counting function
\begin{align*}
\omega(k)
  &=\sum_{p\in\Pc}\Indi{\{p|k\}},
\end{align*}
and the law of $w(J_n)$ with $J_n$  satisfies $(\mathbf{H}_t)$ with $t>0$. In the special case where $J_n$ is uniformly distributed in $[n]$ (which satisfies $(\mathbf{H}_t)$ with $t=1$ by Example \ref{ex}), the celebrated Erd\"os-Kac theorem (see \cite{ErKac}) states that $ (\omega(J_n)-\log\log(n))/\sqrt{\log\log(n)}$ converges in distribution to a standard Gaussian random variable. To assess the rate of convergence of the Erd\"os-Kac theorem, one can use the Wasserstein distance given by 
\begin{align*}
d_W(X,Y)
  &=\sup_{g\in Lip_1}|\E[g(X)]-\E[g(Y)]|,
\end{align*}
where $X, Y$ are real-valued random variables and $Lip_1$ denotes the set of 1-Lipchitz functions.\\

\noindent There are several probabilistic proofs for obtaining assessments of the distance between $\omega(J_n)$ subject to normalization and a standard Gaussian random variable.   We refer to Harper \cite{Harp} for two proofs relying on Stein's method in the Kolmogorov metric. The bound \eqref{eq:GenKub} provides a rather simple scheme for achieving the rate \eqref{e:rateEK} (which is the same as \cite{Harp} and up to the $\log\log\log(n)$ term optimal), not only for the case in which $J_{m}$ is uniform, but for a wide range of variables satisfying $\mathbf{(H_t)}$.

\begin{Theorem}\label{eq:Kub}
Suppose that the law of $J_{n}$ is such that there is a constant $\kappa>0$ independent of $n$, such that $\{J_n\}_{n\geq 1}$ satisfies the condition $\mathbf{(H_t)}$. Then there exists a constant $C>0$, such that 
\begin{align}\label{e:rateEK}
d_W\left(\frac{\omega(J_n)-\log\log(n)}{\sqrt{\log\log(n)}},N\right)
  &\leq C\frac{\log\log\log(n)}{\sqrt{\log\log(n)}},
\end{align}
where $N$ is a standard Gaussian random variable denotes the Wasserstein distance. 
\end{Theorem}

\begin{proof}
We observe that  
\begin{align*}
\omega(J_n)
  &=\sum_{p\in\Pc\cap[1,n]}\Indi{\{{v_p}\geq 1\}}.
\end{align*}
Define the arithmetic function
\begin{align*}
\tilde{\omega}_n(k)
  &:=\sum_{p\in\Pc\cap[1,n^{\frac{1}{3\log\log(n)^2}}]}\Indi{\{p\ \text{ divides } k\}}.
\end{align*}
By Merten's formula (see \cite[Theorem~9]{Ten}), we have that
\begin{align*}
|\sum_{p\in\Pc\cap[1,n]}\frac{1}{p}-\log\log(n)|
  &\leq 1,
\end{align*}
for $n$ large. Combining this inequality with $\mathbf{(H_t)}$, we deduce the existence of a constant $C>0$ such that  
\begin{align}\label{eq:Lqboundwwomegatilde}
\E[\omega(J_n)-\omega_n(J_n)]+\Var[\omega(J_n)-\tilde{\omega}_n(J_n)]
  &\leq C\log\log\log(n).
\end{align}
and 
\begin{align}\label{eq:Lqboundwwomegatilde2}
|\E[\omega(J_n)]-\log\log(n)|+|\Var[\omega(J_n)]-\log\log(n)|
  &\leq C\log\log\log(n).
\end{align}
Observe that 
\begin{multline*}
d_{W}(\omega(J_n)-\E[\omega(J_n)],M_n)\\
\begin{aligned}
  &\leq d_{W}(\omega(J_n)-\E[\omega(J_n)],\tilde{\omega}_n(J_n)-\E[\tilde{\omega}_n(J_n)])\\
	&+ d_{W}(\tilde{\omega}_n(J_n)-\E[\tilde{\omega}_n(J_n)],\tilde{\omega}_n(J_n)-\E[\tilde{\omega}_n(J_n)]\Indi{\{|\tilde{\omega}_n(J_n)-\E[\tilde{\omega}_n(J_n)]|\leq\log\log(n)^2\}})\\
  &+ d_{W}((\tilde{\omega}_n(J_n)-\E[\tilde{\omega}_n(J_n)])\Indi{|\tilde{\omega}_n(J_n)-\E[\tilde{\omega}_n(J_n)]|\leq \log\log(n)^2},\sum_{p\in\Pc}\Indi{\{|p|\leq n^{\frac{1}{3\log\log(n)^2}}\}} g_{p})\\
	&+ d_{W}(\sum_{p\in\Pc}\Indi{\{|p|\leq n^{\frac{1}{3\log\log(n)^2}}\}} g_{p},M_n),
\end{aligned}
\end{multline*}
where the $g_{p}$ are independent random variables with geometric law of parameter $1/p$ and $M_n$ denotes a centered Gaussian random variable with variance $\log\log(n)$.
Using the Berry Esseen estimations, as well as \eqref{eq:Lqboundwwomegatilde} and \eqref{eq:Lqboundwwomegatilde2}, we thus get
\begin{multline}\label{eqref:eqprevtolasteKs}
d_{W}(\omega(J_n)-\E[\omega(J_n)],M_n)\\
\begin{aligned}
  &\leq C\log\log\log(n)+ d_{W}((\tilde{\omega}_n(J_n)-\E[\tilde{\omega}_n(J_n)])\Indi{|\tilde{\omega}_n(J_n)-\E[\tilde{\omega}_n(J_n)]|\leq \log\log(n)^2},\sum_{p\in\Pc}\Indi{\{|p|\leq n^{\frac{1}{3\log\log(n)^2}}\}} g_{p}).
\end{aligned}
\end{multline}
Define 
$$Y_{n}:=(\tilde{\omega}_n(J_n)-\E[\tilde{\omega}_n(J_n)])\Indi{|\tilde{\omega}_n(J_n)-\E[\tilde{\omega}_n(J_n)]|\leq \log\log(n)^2}.$$ 
and 
$$\tilde{Y}_n
  := \sum_{p\in\Pc}\Indi{\{|p|\leq n^{\frac{1}{3\log\log(n)^2}}\}} g_{p}.$$
Observe that 
$Y_{n}$ takes values in the set $\{-\log\log(n)^2,\dots, \log\log(n)^2\}$ and consequently, from a straighforward examination of the Wasserstein distance,
\begin{align*}
d_{W}(Y_n,\tilde{Y}_n)
  &\leq \E[|\tilde{Y}_n|\Indi{\{\log\log(n)^2\leq|\tilde{Y}_n|}\}]+\sum_{k=-\log\log(n)^2}^{\log\log(n)^2}|k| | \Pb[\tilde{Y}_n=k]-\Pb[{Y}_n=k] |\\
	&\leq C+C\log\log(n)^4d_{TV}(Y_{n},\tilde{Y}_n). 
\end{align*}
Observe that by the prime number theorem, $\log(|\Pc\cap[1,n^{\frac{1}{3\log\log(n)^2}}]|)=O(\frac{\log(n)}{3\log\log(n)^2})$. Therefore, by Theorem \ref{t}
\begin{align*}
d_{W}(Y_n,\tilde{Y}_n)
  &\leq \Pb[\tilde{Y}_n\geq\log\log(n)^2]+\sum_{k=-\log\log(n)^2}^{\log\log(n)^2}|k| | \Pb[\tilde{Y}_n=k]-\Pb[{Y}_n=k] |\\
	&\leq C +C\log\log(n)^4e^{-\log\log(n)^4}\leq C^{\prime} . 
\end{align*}
Combining the above inequality with \eqref{eqref:eqprevtolasteKs}, we get 
\begin{align*}
d_{W}(\omega(J_n)-\E[\omega(J_n)],M_n)
  &\leq C\log\log\log(n).
\end{align*}
The above inequality leads to the desired conclusion.
\end{proof}


\subsection{Poisson approximation}

Now we show how our general bound is relevant in the case of Poisson approximation. Consider the counting process 
\begin{align*}
X_n(t) = \sum_{p\in \cP\cap [a_n, a_n^{e^t}]} \1({{v_p}\ge 1}), \quad t\in[0,1]
\end{align*}
where $a_n$ diverges as $n\to\infty$ and there exists $\ep>0$ such that
\begin{align*}
a_n \le e^{\frac{\log(n)/e}{[\log\log(n)]^{1+\ep}}}.
\end{align*} 
Hence the cardinality of $\Gamma_n:=\cP\cap [a_n, a_n^{e^t}]$ satisfies the condition of Theorem \ref{t}. Denote by $\{\wt X_n(t), t\in[0,1]\}$  the process obtained by replacing all the ${v_p}$'s by ${g_p}$'s.  By abuse of language, we continue to use $X_n$ to denote the counting measure induced by $X_n$ and make the same convention for $\tilde X_n$.

%

Recall that the total variation distance between random measures $\eta$ and $\eta'$ is defined as $\dtv(\eta,\eta') = \inf \PP[\eta_1\neq \eta_2]$ where the infimum is taken over all pair $(\eta_1,\eta_2)$ of random measures such that its first marginal is equal in law to $\eta$ and its second marginal equal to $\eta'$. 

\begin{Theorem} 
Let $\eta_n$ be a Poisson point process with the same intensity as that of $\tilde X_n$. Then
\begin{align*}
\dtv(X_n, \eta_n) \le \dtv(\pmb{v}^n, \pmb{g}^n) + \frac{2}{a_n}.
\end{align*}
\end{Theorem}
 \begin{proof}
 Let $(\pmb{v}^n, \pmb{g}^n)$ be the optimal coupling such that $\PP[\pmb{v}^n \neq  \pmb{g}^n ]=\dtv(\pmb{v}^n, \pmb{g}^n)$. Then 
\begin{align*}
\dtv(X_n, \tilde X_n)\le \PP[ X_n \neq \tilde X_n ] = \dtv(\pmb{v}^n, \pmb{g}^n).
\end{align*}
On the other hand, we can obtain an error estimate for marginal distributions of $\tilde X_n$ and that of $\eta_n$. For any fixed $t\in [0,1]$, it follows from classical Poisson limit theorem \cite[Section 3.6]{Durrett} that 
 \begin{align*}
\dtv(\wt X_n(t), \eta_n([0,t])) &\le \sum_{p\in \cP\cap [a_n, a_n^{e^t}]} \PP[g_p\ge 1]^2 \\
&\le \max\Big\{p^{-1}: p\in \cP\cap [a_n, a_n^{e^t}] \Big\} \sum_{p\in \cP\cap [a_n, a_n^{e^t}]} \frac{1}{p} \le \frac{2}{a_n},
 \end{align*}
 where we used Mertens' formula in the last step. 
 \end{proof} 


\appendix

\section{Some elementary probabilistic estimates}

We prove a Bonferonni-type estimates. 
\begin{Lemma}\label{l:bonf_App} 
For any positive odd integer $\ga\le |\Gamma_n|$, 
\begin{align}\label{e:A3}
\sum_{\substack{I\subset \Gamma_n \\ |I|\le \ga}} (-1)^{|I|} \PP[p_{D}^m p_I | J_n] &\le \PP[{A}({D},m)]\le \sum_{\substack{I\subset \Gamma_n \\ |I|\le \ga+1}} (-1)^{|I|}  \PP[p_{D}^m p_I | J_n],\\
\label{e:A4}
\sum_{\substack{I\subset \Gamma_n \\ |I|\le \ga}}  (-1)^{|I|}  \frac{1}{p_{D}^m p_I} &\le \PP[\tilde{{A}}({D},m)]\le   \sum_{\substack{I\subset \Gamma_n \\ |I|\le \ga+1}}  (-1)^{|I|}  \frac{1}{p_{D}^m p_I}.
\end{align}
\end{Lemma}
\begin{proof} Notice that 
\begin{align*}
{A}({D},m) = \{ p_{D}^m | J_n \} \cap \Big( \big(\bigcup_{q\in\Gamma_n\setminus{D}}\{ q| J_n \}\big) \cup \big(\bigcup_{q\in{D}} \{q^{m_q+1}| J_n\}\big)\Big)^c.
\end{align*}
Hence, 
\begin{align*}
\PP[{A}({D},m) ] &= \PP[ p_{D}^m | J_n ] - \PP\big[\big(\bigcup_{q\in\Gamma_n\setminus{D}}\{ qp_{D}^m | J_n \}\big) \cup \big(\bigcup_{q\in{D}} \{q p^m_{D}| J_n\}\big) \big] \\
&= \PP[ p_{D}^m | J_n ] - \PP\big[\bigcup_{q\in\Gamma_n}\{ qp_{D}^m | J_n \}\big]. 
\end{align*}
By Bonferonni inequalities \cite[Exercise 1.6.10]{Durrett}, we have for any positive odd integer $\ga < |\Gamma_n|$ that 
\begin{align*}
 \sum_{\substack{\emptyset\neq I\subset \Gamma_n \\ |I|\le \ga+1}} (-1)^{|I|+1} \PP\big[ \bigcap_{q\in I} \{q p^m_{D} | J_n \} \big] \le  \PP\big[\bigcup_{q\in\Gamma_n}\{ qp_{D}^m | J_n \}\big] \le  \sum_{\substack{\emptyset\neq I\subset \Gamma_n \\ |I|\le \ga}} (-1)^{|I|+1} \PP\big[ \bigcap_{q\in I} \{q p^m_{D} | J_n \} \big].
\end{align*}
Observe that $\cap_{q\in I}\{ qp^m_{D} | J_n\} = \{p^m_{D} p_I | J_n\}$. Thus, 
\begin{align*}
\PP[{A}({D},m) ] &\le \PP[ p_{D}^m | J_n ] - \sum_{\substack{\emptyset\neq I\subset \Gamma_n \\ |I|\le \ga+1}} (-1)^{|I|+1} \PP\big[ \bigcap_{q\in I} \{q p^m_{D} | J_n \} \big]  =\sum_{\substack{I\subset \Gamma_n \\ |I|\le \ga+1}} (-1)^{|I|} \PP[p_{D}^m p_I | J_n],  \\ 
\PP[{A}({D},m) ] &\ge \PP[ p_{D}^m | J_n ] - \sum_{\substack{\emptyset\neq I\subset \Gamma_n \\ |I|\le \ga}} (-1)^{|I|+1} \PP\big[ \bigcap_{q\in I} \{q p^m_{D} | J_n \} \big] = \sum_{\substack{I\subset \Gamma_n \\ |I|\le \ga}} (-1)^{|I|}\PP[p_{D}^m p_I | J_n],
\end{align*}
ending the proof of \eqref{e:A3}. The same argument, with independence and the exact law of ${g_p}$, yields 
\eqref{e:A4}. We leave the details to the interested reader.  
\end{proof}

We record here the Chernoff bound for Poisson random variables.
\begin{Lemma}\label{l:chernoff}
 If $0<\lambda<x$ and $M$ is a Poisson random variable with parameter $\lambda>0$, then 
\begin{align*}
\Pb[M\geq x]
  & \leq e^{-\lambda}(e\lambda)^{x}x^{-x}.
\end{align*}
Equivalently, 
\begin{align*}
\sum_{k\geq x}\frac{\lambda^{k}}{k!}
  &\leq \Big(\frac{e\lambda}{x}\Big)^{x}.
\end{align*}
\end{Lemma}

\noindent{\bf Acknowledgments}. This research was supported by FNR Grant R-AGR-3410-12-Z (MISSILe) from the University of Luxembourg and partially supported by Grant R-146-000-230-114 from the National University of Singapore.

\end{document}